\documentclass[a4paper]{amsart}

\usepackage[utf8]{inputenc}
\usepackage[english]{babel}

\usepackage[T1]{fontenc}
\usepackage{lmodern}  

\usepackage{my-math-symbol}
\usepackage{my-cleveref}
\usepackage{my-equation-numbering}
\usepackage{tikz}
\usetikzlibrary{cd}

\newcommand{\Paneitz}{P}
\newcommand{\dilation}{\delta}

\newcommand{\Hsymbol}{S_{H}}
\newcommand{\Hpsido}{\Psi_{H}}
\newcommand{\HSobolev}{W_{H}}
\newcommand{\ovone}{\bar{1}}

\usepackage[non-sorted-cites,initials,alphabetic,nobysame]{amsrefs}

\AtBeginDocument{%
	\def\MR#1{}
}

\title{CR Paneitz operator and embeddability}
\author{Yuya Takeuchi}
\address{Division of Mathematics \\ Institute of Pure and Applied Sciences \\ University of Tsukuba
	\\ 1- 1- 1 Tennodai, Tsukuba, Ibaraki 305-8571 Japan}
\email{ytakeuchi@math.tsukuba.ac.jp, yuya.takeuchi.math@gmail.com}

\dedicatory{Dedicated to Professor Linda Preiss Rothschild on the occasion of her 80-th birthday.}



\thanks{This work was supported by JSPS KAKENHI Grant Number JP25K17247.}

\begin{document}

\begin{abstract}
	In this article,
	we give a brief survey of recent developments on
	relations between global embeddability of a closed strictly pseudoconvex CR manifold
	and the CR Paneitz operator.
\end{abstract}

\maketitle

\section{Introduction}
\label{section:introduction}

The \emph{CR Paneitz operator} is a CR invariant linear differential operator
with leading term the square of the sub-Laplacian.
This operator was first found by Graham~\cites{Graham1983,Graham1983-2,Graham1984}
on the boundary of the Siegel domain and the ball
as the obstruction of a smooth solution of the Dirichlet problem
with respect to the complex hyperbolic metric.
He called this operator as the ``compatibility operator.''
Graham and Lee~\cite{Graham-Lee1988} generalized this operator
to the boundary of a strictly pseudoconvex domain in $\bbC^{2}$
and obtained its explicit formula in terms of the Tanaka-Webster connection.
Hirachi~\cite{Hirachi1993} observed that this operator appears
in the transformation law of the logarithmic singularity of the \Szego kernel under conformal change,
and proved that it is CR invariant.
Note that this operator is an example of the CR GJMS operators
defined by Gover and Graham~\cite{Gover-Graham2005}.
The name ``CR Paneitz operator'' was originally introduced by Chiu~\cite{Chiu2006},
motivated by its analogy with the Paneitz operator~\cite{Paneitz2008},
a conformally invariant linear differential operator
with leading term the square of the Laplacian.

The CR Paneitz operator is closely related to fundamental problems in CR geometry.
For example,
Chanillo, Chiu, and Yang~\cite{Chanillo-Chiu-Yang2012} have proved that
the non-negativity of the CR Paneitz operator together with positive Tanaka-Webster scalar curvature
implies the embeddability of the CR manifold.
Conversely,
the author~\cite{Takeuchi2020-Paneitz} has shown that
the non-negativity of the CR Paneitz operator follows from the embeddability.
Moreover,
Cheng, Malchiodi, and Yang~\cite{Cheng-Malchiodi-Yang2017} have obtained a positive mass theorem
under the assumptions of the non-negativity of the CR Paneitz operator and positive Tanaka-Webster scalar curvature.
Combining these results yields that
any embeddable CR manifold admits a minimizer of the CR Yamabe functional.
Note that Cheng, Malchiodi, and Yang~\cite{Cheng-Malchiodi-Yang2023} have recently found that
there exists an example of non-embeddable CR manifolds such that
it has no minimizer of the CR Yamabe functional,
which is a crucial difference from the conformal case.
We also remark that any CR manifold admits a contact form
with constant Tanaka-Webster scalar curvature,
which has been proved by Gamara~\cite{Gamara2001} using the theory of critical points at infinity.

The aim of this survey is to explain analytic properties of the CR Paneitz operator.
In particular,
we will focus on relations between this operator and embeddability.
This paper is organized as follows.
In \cref{section:preliminaries},
we recall basic definitions and facts on CR manifolds, Tanaka-Webster connection, and embeddability.
In \cref{section:CR-Paneitz-operator},
we give a definition of the CR Paneitz operator following \cite{Takeuchi2020-Paneitz}
and explain some basic properties of this operator.
\cref{section:non-negativity-of-the-CR-Paneitz-operator-and-embeddability} deals with
relations between the non-negativity of the CR Paneitz operator and the embeddability.
\cref{section:Heisenberg-calculus} gives a brief exposition of the Heisenberg calculus,
which is a theory of pseudodifferential operators reflecting on the geometric structure of a CR manifold.
\cref{section:spectrum-of-the-CR-Paneitz-operator-and-embeddability} is devoted to
a brief exposition of spectral properties of the CR Paneitz operator.
In \cref{section:Rossi-sphere},
we study the CR Paneitz operator on the Rossi sphere via spherical harmonics.

\section{Preliminaries}
\label{section:preliminaries}

\subsection{CR structures}

Let $M$ be a smooth three-dimensional manifold without boundary.
A \emph{CR structure} is a complex line subbundle $T^{1, 0} M$
of the complexified tangent bundle $TM \otimes \bbC$ such that
\begin{equation}
	T^{1, 0} M \cap T^{0, 1} M = 0,
\end{equation}
where $T^{0, 1} M$ is the complex conjugate of $T^{1, 0} M$ in $T M \otimes \bbC$.
Introduce an operator $\delbb \colon C^{\infty}(M) \to \Gamma((T^{0, 1} M)^{\ast})$ by
\begin{equation}
	\delbb f = (d f)|_{T^{0, 1} M}.
\end{equation}
A smooth function $f$ is called a \emph{CR holomorphic function}
if $\delbb f = 0$.
A \emph{CR pluriharmonic function} is a real-valued smooth function
that is locally the real part of a CR holomorphic function.
We denote by $\scrP$ the space of CR pluriharmonic functions.
A CR manifold $(M, T^{1, 0} M)$ is said to be \emph{embeddable}
if there exists a smooth embedding $F = (f_{1}, \dots, f_{N}) \colon M \to \bbC^{N}$
such that each $f_{i}$ is CR holomorphic.

A CR structure $T^{1, 0} M$ is said to be \emph{strictly pseudoconvex}
if there exists a nowhere-vanishing real one-form $\theta$ on $M$
such that
$\theta$ annihilates $T^{1, 0} M$ and
\begin{equation}
	- \sqrt{- 1} d \theta (Z, \ovZ) > 0, \qquad
	0 \neq Z \in T^{1, 0} M;
\end{equation}
we call such a one-form a \emph{contact form}.
The triple $(M, T^{1, 0} M, \theta)$ is called a \emph{pseudo-Hermitian manifold}.
Note that $H M \coloneqq \Ker \theta$ defines a cooriented contact structure on $M$.
Denote by $T$ the \emph{Reeb vector field} with respect to $\theta$;
that is, the unique vector field satisfying
\begin{equation}
	\theta(T) = 1, \qquad T \contr d\theta = 0.
\end{equation}
Let $Z_{1}$ be a local frame of $T^{1, 0} M$,
and set $Z_{\ovone} = \overline{Z_{1}}$.
Then
$(T, Z_{1}, Z_{\ovone})$ gives a local frame of $T M \otimes \bbC$,
called an \emph{admissible frame}.
Its dual frame $(\theta, \theta^{1}, \theta^{\ovone})$
is called an \emph{admissible coframe}.
The two-form $d \theta$ is written as
\begin{equation}
	d \theta = \sqrt{- 1} l_{1 \ovone} \theta^{1} \wedge \theta^{\ovone},
\end{equation}
where $l_{1 \ovone}$ is a positive function.
We use $l_{1 \ovone}$ and its multiplicative inverse $l^{1 \ovone}$
to raise and lower indices.

\subsection{Tanaka-Webster connection}

Let $(M, T^{1, 0} M)$ be a strictly pseudoconvex CR manifold of dimension three.
A contact form $\theta$ induces a canonical connection $\nabla$,
called the \emph{Tanaka-Webster connection} with respect to $\theta$~\cites{Tanaka1962,Webster1978}.
It is defined by
\begin{equation}
	\nabla T
	= 0,
	\quad
	\nabla Z_{1}
	= \omega_{1}^{\ 1} Z_{1},
	\quad
	\nabla Z_{\ovone}
	= \omega_{\ovone}^{\ \ovone} Z_{\ovone}
	\quad
	\rbra*{ \omega_{\ovone}^{\ \ovone}
	= \overline{\omega_{1}^{\ 1}} }
\end{equation}
with the following structure equations:
\begin{gather}
	d \theta^{1}
	= \theta^{1} \wedge \omega_{1}^{\ 1}
	+ A^{1}_{\ \ovone} \theta \wedge \theta^{\ovone}, \\
	d l_{1 \ovone}
	= \omega_{1}^{\ 1} l_{1 \ovone}
	+ l_{1 \ovone} \omega_{\ovone}^{\ \ovone}.
\end{gather}
The tensor $A_{1 1} = \overline{A_{\ovone \ovone}}$
is called the \emph{Tanaka-Webster torsion}.
We say that $\theta$ is torsion-free if $A_{1 1} = 0$.
Note that this condition is equivalent to
that the Reeb vector field $T$ preserves the CR structure $T^{1, 0} M$.
The curvature form
$\Omega_{1}^{\ 1} = d \omega_{1}^{\ 1}$
of the Tanaka-Webster connection satisfies
\begin{equation} \label{eq:curvature-form-of-TW-connection}
	\Omega_{1}^{\ 1}
	\equiv \Scal \cdot l_{1 \ovone} \, \theta^{1} \wedge \theta^{\ovone}
		\qquad \text{modulo } \theta,
\end{equation}
where $\Scal$ is the \emph{Tanaka-Webster scalar curvature}.
We denote the components of a successive covariant derivative of a tensor
by subscripts preceded by a comma,
for example, $K_{1 \ovone , 1}$;
we omit the comma if the derivatives are applied to a function.
We use the index $0$ for the component $T$ or $\theta$ in our index notation.
In this notation,
the operator $\delbb$ is given by
\begin{equation}
\label{eq:tensorial-rep-of-delb}
	\delbb f
	= f_{\ovone} \theta^{\ovone}.
\end{equation}
The commutators of the second derivatives for $f \in C^{\infty}(M)$ are given by
\begin{equation}
\label{eq:commutator-of-covariant-derivatives-on-functions}
	f_{1 \ovone} - f_{\ovone 1}
	= \sqrt{- 1} l_{1 \ovone} f_{0},
	\qquad
	f_{0 1} - f_{1 0}
	= A_{1 1} f^{1};
\end{equation}
see~\cite{Lee1988}*{(2.14)}.
Define the \emph{Kohn Laplacian} $\Box_{b}$ and the \emph{sub-Laplacian} $\Delta_{b}$ by
\begin{equation}
\label{eq:Kohn-Laplacian}
	\Box_{b} f
	\coloneqq - f_{\ovone}^{\ \ovone},
	\qquad
	\Delta_{b} f
	\coloneqq (\Box_{b} + \overline{\Box}_{b}) f
\end{equation}
for $f \in C^{\infty}(M)$.
It follows from \cref{eq:commutator-of-covariant-derivatives-on-functions} that
\begin{equation}
\label{eq:complex-conjugate-of-Kohn-Laplacian}
	\overline{\Box}_{b}
	= \Box_{b} - \sqrt{- 1} T.
\end{equation}
The second derivatives of a $(1, 0)$-form $K = K_{1} \theta^{1}$
satisfy the following commutation relations:
\begin{gather}
\label{eq:commutator-of-covariant-derivatives-on-tensors1}
	K_{1, 1 \ovone} - K_{1, \ovone 1}
	= \sqrt{- 1} l_{1 \ovone} K_{1, 0} + \Scal \cdot l_{1 \ovone} K_{1}, \\
\label{eq:commutator-of-covariant-derivatives-on-tensors2}
	K_{1, 0 \ovone} - K_{1, \ovone 0}
	= A^{1}_{\ \ovone} K_{1, 1} + A^{1}_{\ \ovone, 1} K_{1}
	= (A^{1}_{\ \ovone} K_{1})_{, 1};
\end{gather}
see~\cite{Lee1988}*{Lemma 2.3}.
In particular,
\begin{equation}
	{f_{0 1}}^{1}
	= {f_{1 0}}^{1} + (A_{1 1} f^{1}) {{}_{,}}^{1}
	= {f_{1}}^{1} {}_{0} + (A^{1 1} f_{1}) {}_{, 1}
		+ (A^{\ovone \ovone} f_{\ovone}) {}_{, \ovone}.
\end{equation}
This implies that
\begin{equation}
\label{eq:commutator-of-Kohn-Laplacian-and-its-conjugate}
	\comm{\Box_{b}}{\overline{\Box}_{b}}
	= \comm{\overline{\Box}_{b} + \sqrt{- 1} T}{\overline{\Box}_{b}}
	= \sqrt{- 1} \comm{T}{\overline{\Box}_{b}}
	= \calQ - \overline{\calQ},
\end{equation}
where $\calQ f = \sqrt{- 1} (A^{\ovone \ovone} f_{\ovone}) {}_{, \ovone}$
and $\overline{\calQ}$ is the complex conjugate of $\calQ$.

\subsection{Kohn Laplacian and embeddability}

Let $(M, T^{1, 0} M)$ be a closed three-dimensional strictly pseudoconvex CR manifold
and $\theta$ be a contact form on $M$.
We consider $\Box_{b}$ as an unbounded operator on $L^{2}(M)$ with domain
\begin{equation}
	\Dom \Box_{b}
	= \Set{f \in L^{2}(M) | \text{$\Box_{b} f$ in the weak sense is in $L^{2}(M)$}}.
\end{equation}
This operator is known to be non-negative and self-adjoint.
In particular,
the spectrum $\Spec \Box_{b}$ of $\Box_{b}$ is contained in $\clop{0}{\infty}$.
Moreover,
it is discrete except zero.

\begin{theorem}[\cite{Burns-Epstein1990-Embed}*{Theorem 1.3 and Corollary 1.11};
see also \cite{Hsiao-Marinescu2017}*{Theorem 1.7}]
\label{thm:spectrum-of-Kohn-Laplacian}
	The set $\Spec \Box_{b} \cap \opop{0}{\infty}$ is a discrete subset of $\opop{0}{\infty}$.
	Moreover,
	for any $\lambda \in \Spec \Box_{b} \cap \opop{0}{\infty}$,
	the eigenspace $\Ker (\Box_{b} - \lambda I)$ is a finite dimensional subspace of $C^{\infty}(M)$.
\end{theorem}

It is natural to ask when $0$ is isolated in $\Spec \Box_{b}$,
or equivalently,
$\Box_{b}$ has closed range.
In fact,
this property is equivalent to the embeddability of $(M, T^{1, 0} M)$.

\begin{theorem}[\cite{Burns1979}*{Lemma 1} and \cite{Kohn1986}*{Theorem 5.2}]
\label{thm:Kohn-Laplacian-and-embeddability}
	A three-dimensional closed strictly pseudoconvex CR manifold $(M, T^{1, 0} M)$ is embeddable
	if and only if $\Box_{b}$ has closed range;
	or equivalently,
	$0$ is an isolated point of $\Spec \Box_{b}$.
\end{theorem}

\section{CR Paneitz operator}
\label{section:CR-Paneitz-operator}

Let $(M, T^{1, 0} M)$ be a strictly pseudoconvex CR manifold of dimension three
and $\theta$ be a contact form on $M$.
Then the differential operator
\begin{equation}
\label{eq:definition-of-d^c_CR}
	d^{c}_{\CR} \colon C^{\infty}(M) \to \Omega^{1}(M) ; \quad
	u \mapsto \frac{\sqrt{- 1}}{2} \rbra*{u_{\ovone} \theta^{\ovone}
		- u_{1} \theta^{1}} + \frac{1}{2} (\Delta_{b} u) \theta
\end{equation}
is independent of the choice of $\theta$~\cite{Takeuchi2020-Paneitz}*{Lemma 3.1}.
This operator is a CR analog of $d^{c} = (\sqrt{- 1} / 2) (\delb - \partial)$ in complex geometry.
As in complex geometry,
CR pluriharmonic functions are smooth functions annihilated by $d d^{c}_{\CR}$.

\begin{lemma}[\cite{Takeuchi2020-Paneitz}*{Lemma 3.2}]
\label{lem:characterization-of-CR-pluriharmonic}
	For $u \in C^{\infty}(M)$,
	\begin{equation}
	\label{eq:formula-of-dd^c_CR}
		d d^{c}_{\CR} u
		= (P_{1} u) \, \theta \wedge \theta^{1} + (P_{\ovone} u) \, \theta \wedge \theta^{\ovone},
	\end{equation}
	where
	\begin{equation}
		P_{1} u
		= u_{\ovone \ 1}^{\ \ovone} + \sqrt{- 1} A_{1 1} u^{1},
		\qquad
		P_{\ovone} u
		= u_{1 \ \ovone}^{\ 1} - \sqrt{- 1} A_{\ovone \ovone} u^{\ovone}.
	\end{equation}
	In particular,
	u is CR pluriharmonic if and only if $d d^{c}_{\CR} u = 0$.
\end{lemma}

The \emph{CR Paneitz operator} $\Paneitz$ is defined by
\begin{equation}
	\Paneitz u
	\coloneqq \frac{1}{2} \rbra*{(P_{1} u) {}_{,}^{\ 1} + (P_{\ovone} u) {}_{,}^{\ \ovone}}
	= \frac{1}{2} \rbra*{\Box_{b} \overline{\Box}_{b} + \overline{\Box}_{b} \Box_{b}
		+ \calQ + \overline{\calQ}} u,
\end{equation}
which is a real operator.
It follows from \cref{eq:commutator-of-Kohn-Laplacian-and-its-conjugate} that
\begin{equation}
	\Paneitz
	= \overline{\Box}_{b} \Box_{b} + \calQ
	= \Box_{b} \overline{\Box}_{b} + \overline{\calQ}.
\end{equation}
For $u, v \in C^{\infty}_{c}(M)$,
we have
\begin{align}
	\int_{M} d^{c}_{\CR} u \wedge d d^{c}_{\CR} v
	&= \frac{\sqrt{- 1}}{2} \int_{M} \sbra*{u_{\ovone} (P_{1} v) \theta^{\ovone} \wedge \theta \wedge \theta^{1}
		- u_{\ovone} (P_{1} v) \theta^{1} \wedge \theta \wedge \theta^{\ovone}} \\
	&= \frac{1}{2} \int_{M} \sbra*{u^{1} (P_{1} v) + u^{\ovone} (P_{\ovone} v)} \, \theta \wedge d \theta.
\end{align}
Integration by parts yields that
\begin{equation}
\label{eq:cohomological-expression-of-CR-Paneitz}
	\int_{M} d^{c}_{\CR} u \wedge d d^{c}_{\CR} v
	= - \int_{M} u (\Paneitz v) \, \theta \wedge d \theta.
\end{equation}

\begin{proposition}[\cite{Gover-Graham2005}*{Proposition 4.6} or \cite{Chiu2006}*{Proposition 3.2}]
	The CR Paneitz operator $\Paneitz$ is formally self-adjoint.
\end{proposition}

\begin{proof}
	First note that $\Paneitz$ is a real operator,
	and so it suffices to consider real-valued functions.
	Let $u, v \in C^{\infty}_{c}(M)$.
	Then \cref{eq:cohomological-expression-of-CR-Paneitz} implies that
	\begin{align}
		\int_{M} u (\Paneitz v) \, \theta \wedge d \theta
		&= - \int_{M} d^{c}_{\CR} u \wedge d d^{c}_{\CR} v \\
		&= \int_{M} d \rbra*{d^{c}_{\CR} u \wedge d^{c}_{\CR} v}
			- \int_{M} d d^{c}_{\CR} u \wedge d^{c}_{\CR} v \\
		&= - \int_{M} d^{c}_{\CR} v \wedge d d^{c}_{\CR} u \\
		&= \int_{M} v (\Paneitz u) \, \theta \wedge d \theta,
	\end{align}
	which means that $\Paneitz$ is formally self-adjoint.
\end{proof}

The CR Paneitz operator is said to be \emph{non-negative}
if $\int_{M} u (\Paneitz u) \, \theta \wedge d \theta \geq 0$
for any real-valued $u \in C^{\infty}_{c}(M)$.
\cref{eq:cohomological-expression-of-CR-Paneitz} yields that
this condition is equivalent to
\begin{equation}
	\int_{M} d^{c}_{\CR} u \wedge d d^{c}_{\CR} u \leq 0,
\end{equation}
which is independent of the choice of $\theta$.
This fact also follows from the transformation law of $\Paneitz$ under conformal change.

\begin{lemma}[\cite{Hirachi1993}*{Lemma 7.4}]
	Under the conformal change $\whxth = e^{\Upsilon} \theta$,
	one has $\whP = e^{- 2 \Upsilon} \Paneitz$,
	where $\whP$ is defined in terms of $\whxth$.
\end{lemma}

\begin{proof}
	Since $d^{c}_{\CR}$ is independent of the choice of a contact form,
	we have
	\begin{align}
		\int_{M} u (\Paneitz v) \, \theta \wedge d \theta
		= - \int_{M} d^{c}_{\CR} u \wedge d d^{c}_{\CR} v
		&= \int_{M} u (\whP v) \, \whxth \wedge d \whxth \\
		&= \int_{M} u (e^{2 \Upsilon} \whP v) \, \theta \wedge d \theta
	\end{align}
	for any $u, v \in C^{\infty}_{c}(M)$.
	This yields that $\whP = e^{- 2 \Upsilon} \Paneitz$.
\end{proof}

The non-negativity of the CR Paneitz operator
follows from the existence of a torsion-free contact form.

\begin{proposition}
	Assume that $(M, T^{1, 0} M)$ admits a torsion-free contact form $\theta$.
	Then the CR Paneitz operator is non-negative.
\end{proposition}

\begin{proof}
	We recall some analytic properties of $\Delta_{b}$ and $T$.
	Since $\Delta_{b}$ is a non-negative subelliptic self-adjoint operator,
	its spectrum $\Spec \Delta_{b}$ is a discrete subset of $\clop{0}{\infty}$,
	and the eigenspace $\calH_{a}$ with eigenvalue $a \in \Spec \Delta_{b}$
	is a finite-dimensional subspace of $C^{\infty}(M)$.
	Moreover,
	the vector field $\sqrt{- 1} T$ is formally self-adjoint and commutes with the sub-Laplacian
	since $\theta$ is torsion-free.
	Hence each eigenspace $\calH_{a}$
	is decomposed into the orthogonal direct sum of $\calH_{a, b}$,
	where $b$ is a real number and
	\begin{equation}
		\calH_{a, b}
		= \Set{f \in \calH_{a} | \sqrt{- 1} T f = b f }.
	\end{equation}
	Since $2 \Box_{b} = \Delta_{b} + \sqrt{- 1} T$
	and $2 \overline{\Box}_{b} = \Delta_{b} - \sqrt{- 1} T$
	are non-negative operators,
	we have $a \geq \abs{b}$.
	The restriction of $\Paneitz = \overline{\Box}_{b} \Box_{b} = (\Delta_{b}^{2} - T^{2}) / 4$
	to $\calH_{a, b}$ is given by $(a^{2} - b^{2}) / 4$,
	which is non-negative.
\end{proof}

We note that a CR manifold admitting a torsion-free contact form must be embeddable;
see \cite{Lempert1992}*{Theorem 2.1} or \cite{Marinescu-Yeganefar2007}*{Theorem 5.2}.
In particular,
this theorem follows also from \cref{thm:non-negativity-of-CR-Paneitz-operator} in the next section.

\section{Non-negativity of the CR Paneitz operator and embeddability}
\label{section:non-negativity-of-the-CR-Paneitz-operator-and-embeddability}

In this section,
we discuss relations between the non-negativity of the CR Paneitz operator
and the embeddability of a CR manifold.
We first consider the CR Paneitz operator on an embeddable CR manifold.
In this direction,
Chanillo, Chiu, and Yang~\cite{Chanillo-Chiu-Yang2013}*{Theorem 1.7} have proved that
the non-negativity of the CR Paneitz operator is stable
under sufficiently small deformations of strictly pseudoconvex real hypersurface in $\bbC^{2}$.
Moreover,
Case, Chanillo, and Yang~\cite{Case-Chanillo-Yang2016}*{Theorem 1.3} have shown that
the non-negativity condition is preserved under suitable real-analytic embeddable deformations.
Later,
the author~\cite{Takeuchi2020-Paneitz} has proved that
the CR Paneitz operator is non-negative on any embeddable CR manifold.

\begin{theorem}[\cite{Takeuchi2020-Paneitz}*{Theorem 1.1}]
\label{thm:non-negativity-of-CR-Paneitz-operator}
	Let $(M, T^{1, 0} M)$ be a closed embeddable strictly pseudoconvex CR manifold of dimension three.
	Then the CR Paneitz operator is non-negative,
	and its kernel consists of CR pluriharmonic functions.
\end{theorem}

\begin{proof}[Sketch of proof]
	Without loss of generality,
	we may assume that $M$ is connected.
	It follows from \cite{Harvey-Lawson1975}*{Theorem 10.4} that
	$(M, T^{1, 0} M)$ bounds a two-dimensional strictly pseudoconvex Stein space.
	Lempert~\cite{Lempert1995}*{Theorem 8.1} has proved that,
	in this case,
	$(M, T^{1, 0} M)$ can be realized as the boundary of a strictly pseudoconvex domain $\Omega$
	in a two-dimensional complex projective manifold $X$.
	Take an \emph{asymptotically complex hyperbolic} \Kahler form $\omega_{+}$ on $\Omega$;
	that is,
	a \Kahler form on $\Omega$ that approaches the complex hyperbolic \Kahler form near the boundary.
	Denote by $\Box_{+}$ the $\delb$-Laplacian on $\Omega$ with respect to $\omega_{+}$.

	Fix a real-valued $u \in C^{\infty}(\bdry \Omega)$.
	We obtain from \cite{Epstein-Melrose-Mendoza1991} that
	there exists a $\Box_{+}$-harmonic extension $\wtu \in C^{\infty}(\Omega)$ of $u$.
	This $\wtu$ satisfies
	\begin{equation}
		d^{c} \wtu|_{\bdry \Omega}
		= d^{c}_{\CR} u,
		\qquad
		d d^{c} \wtu|_{\bdry \Omega}
		= d d^{c}_{\CR} u.
	\end{equation}
	Moreover,
	$d d^{c} \wtu \wedge \omega_{+} = (\Box_{+} \wtu / 2) \omega_{+}^{2} = 0$ yields that
	\begin{equation}
		d d^{c} \wtu \wedge d d^{c} \wtu
		\leq 0
	\end{equation}
	with equality if and only if $\wtu$ is pluriharmonic.
	It follows from Stokes' theorem and \cref{eq:cohomological-expression-of-CR-Paneitz} that
	\begin{align}
		0
		\geq \int_{\Omega} d d^{c} \wtu \wedge d d^{c} \wtu
		= \int_{\Omega} d \rbra*{d^{c} \wtu \wedge d d^{c} \wtu}
		&= \int_{\bdry \Omega} d^{c} \wtu \wedge d d^{c} \wtu \\
		&= \int_{\bdry \Omega} d^{c}_{\CR} u \wedge d d^{c}_{\CR} u \\
		&= - \int_{\bdry \Omega} u (\Paneitz u) \, \theta \wedge d \theta.
	\end{align}
	This means that $\Paneitz$ is non-negative.
	If $\Paneitz u = 0$,
	then $d d^{c}_{\CR} u = d d^{c} \wtu|_{\bdry \Omega} = 0$,
	and so $u \in \scrP$ according to \cref{lem:characterization-of-CR-pluriharmonic}.
\end{proof}

See \cite{Takeuchi2020-Paneitz} for some applications of this theorem;
e.g., the CR Yamabe problem and the logarithmic singularity of the \Szego kernel.
In the proof of \cref{thm:non-negativity-of-CR-Paneitz-operator},
we use the fact that an embeddable CR manifold can be filled by a \Kahler manifold.
On the other hand,
\cref{thm:Kohn-Laplacian-and-embeddability} states that
the embeddability is equivalent to the closedness of $\Ran \Box_{b}$,
which is a spectral property of $\Box_{b}$.
It is natural to ask whether we can derive $\Paneitz \geq 0$
directly from the fact that $\Ran \Box_{b}$ is closed.

\begin{problem}
	Can we obtain the non-negativity of the CR Paneitz operator
	from the fact that $\Box_{b}$ has closed range?
\end{problem}

As we will see in \cref{section:Rossi-sphere},
the CR Paneitz operator on a non-embeddable CR manifold may have a negative eigenvalue.
However,
it is not known whether its spectrum is bounded from below.

\begin{problem}
	Is the spectrum $\Spec \Paneitz$ of $\Paneitz$ bounded from below
	even if $(M, T^{1, 0} M)$ is non-embeddable?
\end{problem}

Conversely,
the non-negativity of the CR Paneitz operator together with positive scalar curvature
implies the embeddablity.
The proof is based on the Bochner type formula
for the Kohn Laplacian~\cite{Chanillo-Chiu-Yang2012}*{Proposition 2.1};
we remark that the formula is valid also in higher dimensions~\cite{Chanillo-Chiu-Yang2012}*{Remark 2.3}.

\begin{theorem}[\cite{Chanillo-Chiu-Yang2012}*{Theorem 1.4}]
	Let $(M, T^{1, 0} M, \theta)$ be a closed pseudo-Hermitian manifold of dimension three
	such that $\Paneitz$ is non-negative and $\Scal > 0$.
	Then $(M, T^{1, 0} M)$ is embeddable.
\end{theorem}

\begin{proof}
	Take a constant $C > 0$ such that $\Scal \geq C$.
	According to \cref{thm:spectrum-of-Kohn-Laplacian,thm:Kohn-Laplacian-and-embeddability},
	it suffices to show that any non-zero eigenvalue $\lambda$ of $\Box_{b}$
	satisfies $\lambda \geq C / 2$.
	Let $\varphi$ be an eigenfunction of $\Box_{b}$ corresponding to $\lambda$;
	note that $\varphi \in C^{\infty}(M)$.
	Integration by parts gives that
	\begin{equation}
	\label{eq:eigenfunction-of-Kohn-Laplacian}
		\int_{M} \varphi_{\ovone} \overline{\varphi}^{\ovone} \, \theta \wedge d \theta
		= - \int_{M} \varphi_{\ovone}^{\ \ovone} \cdot \overline{\varphi} \, \theta \wedge d \theta
		= \lambda \norm{\varphi}_{L^{2}}^{2}.
	\end{equation}
	We obtain from \cref{eq:commutator-of-covariant-derivatives-on-functions} that
	\begin{align}
		P_{\ovone} \varphi
		= \varphi_{1 \ \ovone}^{\ 1} - \sqrt{- 1} A_{\ovone \ovone} \varphi^{\ovone}
		&= \varphi_{\ovone \ \ovone}^{\ \ovone}
			+ \sqrt{- 1} \rbra*{\varphi_{0 \ovone} - A_{\ovone \ovone} \varphi^{\ovone}} \\
		&= - \lambda \varphi_{\ovone} + \sqrt{- 1} \varphi_{\ovone 0}.
	\end{align}
	This and \cref{eq:commutator-of-covariant-derivatives-on-tensors1} yield that
	\begin{equation}
		\varphi_{\ovone \ovone}^{\ \ \, \ovone}
		= \varphi_{\ovone \ \ovone}^{\ \ovone} - \sqrt{- 1} \varphi_{\ovone 0} + \Scal \cdot \varphi_{\ovone}
		= - 2 \lambda \varphi_{\ovone} - P_{\ovone} \varphi + \Scal \cdot \varphi_{\ovone}.
	\end{equation}
	Hence we have
	\begin{align}
		(\varphi_{\ovone} \overline{\varphi}^{\ovone}) {}_{\ovone}^{\ \, \ovone}
		&= {\varphi_{\ovone \ovone}}^{\ovone} \overline{\varphi}^{\ovone}
			+ \varphi_{\ovone \ovone} \overline{\varphi}^{\ovone \ovone}
			+ \varphi_{\ovone}^{\ \ovone} \overline{\varphi}^{\ovone}_{\ \ovone}
			+ \varphi_{\ovone} \overline{\varphi}^{\ovone \ \ovone}_{\ \ovone} \\
		&= \varphi_{\ovone \ovone} \overline{\varphi}^{\ovone \ovone}
			+ \lambda^{2} \abs{\varphi}^{2} - 3 \lambda \varphi_{\ovone} \overline{\varphi}^{\ovone}
			- (P_{\ovone} \varphi) \overline{\varphi}^{\ovone}
			+ \Scal \cdot \varphi_{\ovone} \overline{\varphi}^{\ovone}.
	\end{align}
	Integration by parts and \cref{eq:eigenfunction-of-Kohn-Laplacian} imply that
	\begin{align}
		0
		&= \int_{M} \sbra*{\varphi_{\ovone \ovone} \overline{\varphi}^{\ovone \ovone}
			+ \lambda^{2} \abs{\varphi}^{2} - 3 \lambda \varphi_{\ovone} \overline{\varphi}^{\ovone}
			- (P_{\ovone} \varphi) \overline{\varphi}^{\ovone}
			+ \Scal \cdot \varphi_{\ovone} \overline{\varphi}^{\ovone}} \, \theta \wedge d \theta \\
		&\geq - 2 \lambda^{2} \norm{\varphi}_{L^{2}}^{2}
			+ \int_{M} (P \varphi) \overline{\varphi} \, \theta \wedge d \theta
			+ C \int_{M} \varphi_{\ovone} \overline{\varphi}^{\ovone} \, \theta \wedge d \theta \\
		&\geq - 2 \lambda^{2} \norm{\varphi}_{L^{2}}^{2}
			+ \lambda C \norm{\varphi}_{L^{2}}^{2} \\
		&= \lambda (C - 2 \lambda) \norm{\varphi}_{L^{2}}^{2}.
	\end{align}
	Since $\lambda > 0$ and $\norm{\varphi}_{L^{2}}^{2} > 0$,
	we have $\lambda \geq C / 2$.
\end{proof}

It is not known whether the additional assumption $\Scal > 0$
is necessary in the above theorem.

\begin{problem}
	Does there exist a non-embeddable pseudo-Hermitian manifold
	such that $\Paneitz \geq 0$ and $\Scal \leq 0$?
\end{problem}

\section{Heisenberg calculus}
\label{section:Heisenberg-calculus}

In this section,
we recall basic properties of Heisenberg pseudodifferential operators;
see~\cites{Beals-Greiner1988,Ponge2008-Book}
for a comprehensive introduction to the Heisenberg calculus.
Throughout this section,
we fix a closed pseudo-Hermitian manifold $(M, T^{1, 0} M, \theta)$ of dimension three.
Set
\begin{equation}
	\frakg M
	\coloneqq (T M / H M) \oplus H M,
\end{equation}
which is a bundle of two-step nilpotent Lie algebras.
The dilation $\dilation_{r}$ on $\frakg M$ for $r > 0$ is defined by
\begin{equation}
	\dilation_{r} |_{T M / H M}
	\coloneqq r^{2},
	\qquad
	\dilation_{r} |_{H M}
	\coloneqq r.
\end{equation}
For $m \in \bbZ$,
the space $\Hsymbol^{m}(M)$
consists of functions in $C^{\infty}((\frakg M)^{\ast} \setminus \{0\})$
that are homogeneous of degree $m$ on each fiber.
This space has a bilinear product
\begin{equation}
	\ast \colon \Hsymbol^{m_{1}}(M) \times \Hsymbol^{m_{2}}(M)
	\to \Hsymbol^{m_{1} + m_{2}}(M).
\end{equation}

For $m \in \bbZ$,
denote by $\Hpsido^{m}(M)$
the space of \emph{Heisenberg pseudodifferential operators
$A \colon C^{\infty}(M) \to C^{\infty}(M)$ of order $m$}.
Roughly speaking,
a differential operator is of order $m$
if it can be written locally as a linear combination of operators
obtained by composing at most $m$ sections of $H M$.
In particular,
the Reeb vector field $T$ is not of order $1$ but of order $2$
although it is a vector field.
The space $\Hpsido^{m}(M)$ is closed under complex conjugate, transpose, and formal adjoint%
~\cite{Ponge2008-Book}*{Proposition 3.1.23}.
In particular,
any $A \in \Hpsido^{m}(M)$ extends to a linear operator
\begin{equation}
	A \colon \scrD^{\prime}(M) \to \scrD^{\prime}(M),
\end{equation}
where $\scrD^{\prime}(M)$ is the space of distributions on $M$.
Moreover,
$\bigcup_{m \in \bbZ} \Hpsido^{m}(M)$ defines a filtered algebra
and $\Hpsido^{- \infty}(M) \coloneqq \bigcap_{m \in \bbZ} \Hpsido^{m}(M)$
coincides with the space of smoothing operators on $M$.
As in the usual pseudodifferential calculus,
there exists the Heisenberg principal symbol,
which has some good properties:

\begin{proposition}[\cite{Ponge2008-Book}*{Propositions 3.2.6 and 3.2.9}]
\label{prop:Heisenberg-principal-symbol}
	(i) The Heisenberg principal symbol $\sigma_{m}$ gives the following exact sequence:
	\begin{equation}
		0 \to \Hpsido^{m - 1}(M) \hookrightarrow \Hpsido^{m}(M)
			\xrightarrow{\sigma_{m}} \Hsymbol^{m}(M) \to 0.
	\end{equation}

	(ii) For $A_{1} \in \Hpsido^{m_{1}}(M)$ and $A_{2} \in \Hpsido^{m_{2}}(M)$,
	one has $\sigma_{m_{1} + m_{2}}(A_{1} A_{2}) = \sigma_{m_{1}}(A_{1}) \ast \sigma_{m_{2}}(A_{2})$.
\end{proposition}

Next we consider approximate inverses of Heisenberg pseudodifferential operators.
We write $A \sim A^{\prime}$ if $A - A^{\prime}$ is a smoothing operator.
Let $A \in \Hpsido^{m}(M)$.
An operator $B \in \Hpsido^{- m}(M)$ is called a \emph{parametrix} of $A$
if $A B \sim I$ and $B A \sim I$.
The existence of a parametrix of a Heisenberg pseudodifferential operator
is determined only by its Heisenberg principal symbol.

\begin{proposition}[\cite{Ponge2008-Book}*{Proposition 3.3.1}]
\label{prop:equivalent-conditions-for-existence-of-parametrix}
	Let $A \in \Hpsido^{m}(M)$
	with Heisenberg principal symbol $a \in \Hsymbol^{m}(M)$.
	Then the following are equivalent:
	\begin{enumerate}
		\item $A$ has a parametrix;
		\item there exists $B \in \Hpsido^{- m}(M)$ such that
			$A B - I, B A - I \in \Hpsido^{- 1}(M)$;
		\item there exists $b \in \Hsymbol^{-m}(M)$ such that
			$a \ast b = b \ast a = 1$.
	\end{enumerate}
\end{proposition}

Now consider the Heisenberg differential operator $\Delta_{b} + I$ of order $2$.
It is known that this operator has a parametrix;
see the proof of~\cite{Ponge2008-Book}*{Proposition 3.5.7} for example.
Since $\Delta_{b} + I$ is positive and self-adjoint,
the $k / 2$-th power $(\Delta_{b} + I)^{k / 2}$ of $\Delta_{b} + I$,
$k \in \bbZ$,
is a Heisenberg pseudodifferential operator of order $k$~\cite{Ponge2008-Book}*{Theorems 5.3.1 and 5.4.10}.
Using this operator,
we define
\begin{equation}
	\HSobolev^{k}(M)
	:= \Set{ u \in \scrD^{\prime}(M) \mid (\Delta_{b} + I)^{k / 2} u \in L^{2}(M) }.
\end{equation}
This space is a Hilbert space with the inner product
\begin{equation}
	\iproduct{u}{v}_{k}
	= \iproduct{(\Delta_{b} + I)^{k / 2} u}{(\Delta_{b} + I)^{k / 2} v}_{L^{2}}.
\end{equation}
The space $C^{\infty}(M)$ is dense in $\HSobolev^{k}(M)$ for any $k \in \bbZ$,
and we have
\begin{equation}
	C^{\infty}(M) = \bigcap_{k \in \bbZ} \HSobolev^{k}(M),
	\qquad
	\scrD^{\prime}(M) = \bigcup_{k \in \bbZ} \HSobolev^{k}(M)
\end{equation}
as topological vector spaces~\cite{Ponge2008-Book}*{Proposition 5.5.3}.
Remark that
the Hilbert space $\HSobolev^{k}(M)$ coincides with the Folland-Stein space $S^{k, 2}(M)$
as a topological vector space~\cite{Ponge2008-Book}*{Proposition 5.5.5}.
Heisenberg pseudodifferential operators act on these Hilbert spaces as follows:

\begin{proposition}[\cite{Ponge2008-Book}*{Propositions 5.5.8} and \cite{Takeuchi2023-GJMS}*{Proposition 4.6}]
\label{prop:mapping-properties-of-Hpsido}
	Any $A \in \Hpsido^{m}(M)$ extends to a continuous linear operator
	\begin{equation}
		A \colon \HSobolev^{k + m}(M) \to \HSobolev^{k}(M)
	\end{equation}
	for every $k \in \bbZ$.
	In particular if $m < 0$,
	the operator $A \colon L^{2}(M) \to L^{2}(M)$ is compact.
\end{proposition}

\section{Spectrum of the CR Paneitz operator and embeddability}
\label{section:spectrum-of-the-CR-Paneitz-operator-and-embeddability}

Let $(M, T^{1, 0} M)$ be a closed (not necessarily embeddable)
strictly pseudoconvex CR manifold of dimension three
and $\theta$ be a contact form on $M$.
Since $\Box_{b}$ does not have closed range in general,
there does not exist the partial inverse of $\Box_{b}$.
However,
we have a \emph{formal} partial inverse of $\Box_{b}$
and a \emph{formal} orthogonal projection to $\Ker \Box_{b}$.

\begin{theorem}[\cite{Beals-Greiner1988}*{Proposition 25.4 and Corollaries 25.64 and 25.67}]
\label{thm:approximate-Szego-projection}
	There exist $S \in \Hpsido^{0}(M)$ and $N \in \Hpsido^{- 2}(M)$ satisfying
	\begin{equation}
		\Box_{b} N + S
		\sim N \Box_{b} + S
		\sim I, 
		\qquad
		S
		\sim S^{\ast}
		\sim S^{2},
		\qquad
		\Box_{b} S
		\sim S \Box_{b}
		\sim 0,
		\qquad
		\delb_{b} S
		\sim 0.
	\end{equation}
\end{theorem}

Consider the CR Paneitz operator $\Paneitz$ on $(M, T^{1, 0} M, \theta)$.
Set $\Pi_{0} \coloneqq S + \ovS \in \Hpsido^{0}(M)$
and $G_{0} \coloneqq N \ovN \in \Hpsido^{- 4}(M)$.
Then we have
\begin{equation}
\label{eq:first-approximation-of-partial-inverse}
	R_{0}
	\coloneqq \Paneitz G_{0} + \Pi_{0} - I \in \Hpsido^{- 1}(M).
\end{equation}
Moreover,
$I + R_{0}$ has a parametrix $A_{0}$ satisfying $A_{0} - I \in \Hpsido^{- 1}(M)$%
~\cite{Takeuchi2024-preprint}*{Lemma 4.5}.
Set
\begin{equation}
	\Pi_{\infty}
	\coloneqq \Pi_{0} A_{0} \in \Hpsido^{0}(M),
	\qquad
	G_{\infty}
	\coloneqq (I - \Pi_{\infty}) G_{0} A_{0}
	\in \Hpsido^{- 4}(M).
\end{equation}
These satisfy
\begin{equation}
	G_{\infty} \Paneitz + \Pi_{\infty}
	\sim \Paneitz G_{\infty} + \Pi_{\infty}
	\sim I,
	\qquad
	\Pi_{\infty}^{\ast}
	\sim \Pi_{\infty}^{2}
	\sim \Pi_{\infty},
	\qquad
	\Pi_{\infty} \Paneitz
	\sim \Paneitz \Pi_{\infty}
	\sim 0;
\end{equation}
see \cite{Takeuchi2024-preprint}*{Proposition 4.6}.

Consider $\Paneitz$ as an unbounded closed operator acting on $L^{2}(M)$
by the maximal closed extension;
that is,
the domain $\Dom \Paneitz$ of $\Paneitz$ is defined by
\begin{equation}
	\Dom \Paneitz
	\coloneqq \Set{u \in L^{2}(M) | \text{$\Paneitz u$ in the weak sense is in $L^{2}(M)$}}.
\end{equation}
We can show that $u - \Pi_{\infty} u \in \HSobolev^{4}(M)$ for any $u \in \Dom \Paneitz$;
in particular,
$\Dom \Paneitz = \Ran \Pi_{\infty} + \HSobolev^{4}(M)$~\cite{Takeuchi2024-preprint}*{Lemma 4.7}.
Since $\Paneitz$ is symmetric on $\HSobolev^{4}(M)$,
we have that $\Paneitz$ is self-adjoint~\cite{Takeuchi2024-preprint}*{Theorem 1.1}.

Let $E$ be the resolution of the identity for $\Paneitz$ and fix $\lambda \geq 0$.
Set
\begin{equation}
	\pi_{\lambda} \coloneqq E(\clcl{- \lambda}{\lambda})
	\colon L^{2}(M) \to \Dom \Paneitz.
\end{equation}
This is an orthogonal projection of $L^{2}(M)$ and satisfies
$\pi_{\lambda} \Paneitz = \Paneitz \pi_{\lambda}$ on $\Dom \Paneitz$.
If $\lambda > 0$,
\begin{equation}
	N_{\lambda} \coloneqq \int_{\bbR} t^{- 1} \chi_{\clcl{- \lambda}{\lambda}^{c}}(t) \, d E(t)
	\colon L^{2}(M) \to \Dom \Paneitz
\end{equation}
is a continuous self-adjoint operator and satisfies
\begin{equation}
	\Paneitz N_{\lambda} + \pi_{\lambda} = I \text{\ on } L^{2}(M),
	\qquad
	N_{\lambda} \Paneitz + \pi_{\lambda} = I \text{\ on } \Dom \Paneitz.
\end{equation}
We can show that $\Paneitz \pi_{\lambda} \sim 0$ for $\lambda > 0$~\cite{Takeuchi2024-preprint}*{Theorem 4.8}.
This implies that $\Paneitz \pi_{\lambda}$ is a compact self-adjoint operator.
It follows from this fact that
$\Spec \Paneitz \pi_{\lambda}$ is discrete except $0$,
and $\Ker (\Paneitz \pi_{\lambda} - \mu I)$ is a finite dimensional subspace of $C^{\infty}(M)$
for any $\mu \neq 0$.
On the other hand,
$\Spec \Paneitz \cap \clcl{- \lambda}{\lambda} = \Spec \Paneitz \pi_{\lambda}$ and
\begin{equation}
	\Ker (\Paneitz - \mu I)
	= \Ker (\Paneitz \pi_{\lambda} - \mu I)
\end{equation}
for any $0 < \abs{\mu} \leq \lambda$.
This gives the following

\begin{theorem}[\cite{Takeuchi2024-preprint}*{Theorem 1.2}]
\label{thm:spectrum-of-CR-Paneitz}
	The set $\Spec \Paneitz \setminus \{0\}$ is a discrete subset of $\bbR \setminus \{0\}$
	and consists only of eigenvalues of finite multiplicity.
	Moreover,
	any eigenfunction of $\Paneitz$ corresponding to each non-zero eigenvalue is smooth.
\end{theorem}

Moreover,
$\pi_{\lambda} \in \Hpsido^{0}(M)$ and $N_{\lambda} \in \Hpsido^{- 4}(M)$ for any $\lambda > 0$%
~\cite{Takeuchi2024-preprint}*{Theorem 4.9};
note that $\pi_{0}$ is not necessarily a Heisenberg pseudodifferential operator%
~\cite{Takeuchi2024-preprint}*{Remark 5.8}.

In the reminder of this section,
we assume that $(M, T^{1, 0} M)$ is embeddable,
or equivalently,
$\Box_{b}$ has closed range.
In this case,
we can take as $S$ the orthogonal projection to $\Ker \Box_{b}$,
known as the \emph{\Szego projection},
and as $N$ the partial inverse of $\Box_{b}$~\cite{Beals-Greiner1988}*{Theorem 25.20}.
Note that $\Paneitz S = \Paneitz \ovS = 0$.
It follows from \cite{Hsiao2015}*{Lemma 4.2} that
\begin{equation}
	\Pi_{0}^{2}
	= (S + \ovS)^{2}
	= S + \ovS + S \ovS + \ovS S
	\sim \Pi_{0}.
\end{equation}
Composing $\Pi_{0}$ to \cref{eq:first-approximation-of-partial-inverse} from the left gives
\begin{equation}
	\Pi_{0} R_{0}
	= \Pi_{0} \Paneitz G_{0} + \Pi_{0}^{2} - \Pi_{0}
	\sim 0.
\end{equation}
It follows from $\Pi_{\infty} \sim \Pi_{\infty}^{\ast} = A_{0}^{\ast} \Pi_{0}$ that
\begin{align}
	\Pi_{\infty}
	\sim A_{0}^{\ast} \Pi_{0}
	\sim A_{0}^{\ast} \Pi_{0} (I + R_{0})
	\sim \Pi_{\infty} (I + R_{0})
	= \Pi_{0} A_{0} (I + R_{0})
	\sim \Pi_{0}.
\end{align}

\begin{theorem}[\cite{Hsiao2015}*{Theorem 4.7}]
\label{thm:closed-range-of-CR-Paneitz-operator}
	If $(M, T^{1, 0} M)$ is embeddable,
	then $\Paneitz$ has closed range
	and $\pi_{0} \in \Hpsido^{0}(M)$.
\end{theorem}

\begin{proof}
	Set
	\begin{equation}
		R_{\infty}
		\coloneqq G_{\infty} \Paneitz + \Pi_{0} - I
		= (G_{\infty} \Paneitz + \Pi_{\infty} - I) - (\Pi_{\infty} - \Pi_{0})
		\in \Hpsido^{- \infty}(M).
	\end{equation}
	Fix $\lambda > 0$.
	It follows from $\Ran S, \Ran \ovS \subset \Ran \pi_{0} = \Ker \Paneitz \subset \Ran \pi_{\lambda}$ that
	\begin{equation}
		E(\clcl{- \lambda}{\lambda} \setminus \{0\})
		= \pi_{\lambda} - \pi_{0}
		= G_{\infty} \Paneitz \pi_{\lambda} - R_{\infty} (\pi_{\lambda} - \pi_{0}).
	\end{equation}
	Since the right hand side maps $L^{2}(M)$ to $C^{\infty}(M)$ continuously,
	the orthogonal projection $E(\clcl{- \lambda}{\lambda} \setminus \{0\})$ is a compact operator,
	and so it is of finite rank.
	Hence there exists $\varepsilon > 0$ such that
	$\Spec \Paneitz \cap \clcl{- \varepsilon}{\varepsilon} = \{0\}$,
	which means that $\Paneitz$ has closed range.
	Moreover,
	this yields that $\pi_{0} = \pi_{\varepsilon} \in \Hpsido^{0}(M)$.
\end{proof}

Summarizing the results in this section and \cref{section:non-negativity-of-the-CR-Paneitz-operator-and-embeddability} yields the following

\begin{theorem}
\label{thm:CR-Paneitz-operator-on-embeddable}
	Assume that $(M, T^{1, 0} M)$ is embeddable.
	Then $\Spec \Paneitz$ is a discrete subset of $\clop{0}{\infty}$.
	Moreover,
	the eigenspace of $\Paneitz$ corresponding to each non-zero eigenvalue
	is a finite dimensional subspace of $C^{\infty}(M)$.
	Furthermore,
	$\scrP$ is dense in $\Ker \Paneitz$.
\end{theorem}

\begin{proof}
	The first and second statements follow from \cref{thm:non-negativity-of-CR-Paneitz-operator,%
	thm:spectrum-of-CR-Paneitz,thm:closed-range-of-CR-Paneitz-operator}.
	It remains to show that $\scrP$ is dense in $\Ker \Paneitz$.
	Let $u \in \Ker \Paneitz$ and take $u_{j} \in C^{\infty}(M)$
	such that $u_{j} \to u$ in $L^{2}(M)$ as $j \to + \infty$.
	Since $\pi_{0}$ maps $C^{\infty}(M)$ to itself
	and $L^{2}(M)$ to itself continuously,
	$\pi_{0} u_{j} \in \Ker \Paneitz \cap C^{\infty}(M) = \scrP$
	and $\pi_{0} u_{j} \to \pi_{0} u = u$.
	This shows $u \in \overline{\scrP}$.
\end{proof}

The proof of \cref{thm:closed-range-of-CR-Paneitz-operator}
depends crucially on the existence of the partial inverse of $\Box_{b}$,
which is equivalent to the embeddability of $(M, T^{1, 0} M)$.
It is natural to ask what happens when $(M, T^{1, 0} M)$ is non-embeddable.

\begin{problem}
	Assume that $(M, T^{1, 0} M)$ is non-embeddable.
	Does $\Paneitz$ have closed range?
	Equivalently,
	is $0$ an isolated point of $\Spec \Paneitz$?
\end{problem}

Case, Chanillo, and Yang~\cite{Case-Chanillo-Yang2015} have considered the supplementary space
\begin{equation}
	\scrW
	\coloneqq \Ker \Paneitz \cap \scrP^{\perp}.
\end{equation}
\cref{thm:CR-Paneitz-operator-on-embeddable} implies that $\scrW = 0$ for any embeddable CR manifold.
It is reasonable to investigate the supplementary space in the non-embeddable case.

\begin{problem}
	Assume that $(M, T^{1, 0} M)$ is non-embeddable.
	Is the supplementary space $\scrW$ finite dimensional?
\end{problem}

\section{Example: Rossi sphere}
\label{section:Rossi-sphere}

In this section,
we consider the Rossi sphere,
which is a closed homogeneous strictly pseudoconvex CR manifold of dimension three.
The unit sphere
\begin{equation}
	S^{3}
	\coloneqq \Set{(z, w) \in \bbC^{2} | \abs{z}^{2} + \abs{w}^{2} = 1}
\end{equation}
has the canonical CR structure $T^{1, 0} S^{3}$.
This CR structure is spanned by
\begin{equation}
	Z_{1}
	\coloneqq \ovw \frac{\del}{\del z} - \ovz \frac{\del}{\del w}.
\end{equation}
A canonical contact form $\theta$ on $S^{3}$ is given by
\begin{equation}
	\theta
	\coloneqq \frac{\sqrt{- 1}}{2} (z d \ovz + w d \ovw
		- \ovz d z - \ovw d w)|_{S^{3}}.
\end{equation}
Note that $(S^{3}, T^{1, 0} S^{3})$ is embeddable by definition.
For a real number $0 < \abs{t} < 1$,
set
\begin{equation}
	Z_{1}(t)
	\coloneqq Z_{1} + t Z_{\ovone}.
\end{equation}
This complex vector field satisfies
\begin{equation}
	\theta(Z_{1}(t)) = 0,
	\qquad
	d \theta(Z_{1}(t), \overline{Z_{1}(t)})
	= 1 - t^{2} > 0.
\end{equation}
This implies that
\begin{equation}
	(S^{3}_{t}, T^{1, 0} S^{3}_{t})
	\coloneqq (S^{3}, \bbC Z_{1}(t))
\end{equation}
is a strictly pseudoconvex CR manifold,
known as the \emph{Rossi sphere}~\cite{Rossi1965},
and $\theta$ is a contact form on $(S^{3}_{t}, T^{1, 0} S^{3}_{t})$.
It is known that any CR holomorphic function on $(S^{3}_{t}, T^{1, 0} S^{3}_{t})$ must be even;
see \cite{Chen-Shaw2001}*{Theorem 12.4.1} for example.
In particular,
$(S^{3}_{t}, T^{1, 0} S^{3}_{t})$ is non-embeddable.
Note that $(\rps^{3}_{t} \coloneqq S^{3}_{t} / \{\pm 1\}, T^{1, 0} \rps^{3}_{t})$
can be embedded to $\bbC^{3}$~\cite{Chen-Shaw2001}*{Chapter 12.4}.

Recently,
Abbas, Brown, Ramasami, and Zeytuncu~\cite{Abbas-Brown-Ramasami-Zeytuncu2019}*{Theorem 5.7}
have proved that $0$ is an accumulation point of the Kohn Laplacian
on $(S^{3}_{t}, T^{1, 0} S^{3}_{t}, \theta)$.
This and \cref{thm:Kohn-Laplacian-and-embeddability} give another proof that
$(S^{3}_{t}, T^{1, 0} S^{3}_{t})$ is non-embeddable.
By a similar argument to theirs,
we would like to show that the CR Paneitz operator $\Paneitz(t)$ on $(S^{3}_{t}, T^{1, 0} S^{3}_{t}, \theta)$
has infinitely many negative eigenvalues.

To this end,
we first recall some facts on spherical harmonics.
We denote by $\scrP_{p, q}(\bbC^{2})$
the space of complex homogeneous polynomials of bidegree $(p, q)$
and by $\scrH_{p, q}(\bbC^{2})$ the space of harmonic $f \in \scrP_{p, q}(\bbC^{2})$.
The restriction map
\begin{equation}
	|_{S^{3}} \colon \scrH_{p, q}(\bbC^{2}) \to C^{\infty}(S^{3})
\end{equation}
is injective since any $f \in \scrH_{p,q}$ is harmonic.
The image of $\scrH_{p, q}(\bbC^{2})$ under this map
is written as $\scrH_{p, q}(S^{3})$.
It is known that the Hilbert space $L^{2}(S^{3})$ has the orthogonal decomposition
\begin{equation}
	L^{2}(S^{3})
	= \bigoplus_{p, q} \scrH_{p, q}(S^{3}).
\end{equation}
Moreover,
$\dim \scrH_{p, q}(S^{3}) = p + q + 1$;
in particular,
$\scrH_{p, q}(S^{3})$ contains a non-zero element.

For each $k \in \bbZ_{> 0}$,
take $v_{k, 1} \in \scrH_{2 k - 1, 0}$ with $\norm{v_{k, 1}}_{L^{2}} = 1$.
We set $c_{k}(l) \coloneqq (l - 2) (2 k - l + 2)$ and
\begin{equation}
	v_{k, i}
	\coloneqq \rbra*{\prod_{l = 1}^{i - 1} \frac{1}{\sqrt{c_{k}(2 l + 1) c_{k}(2 l + 2)}}}
		(Z_{1})^{2 i - 2} v_{k, 1}
		\in \scrH_{2 k - 2 i + 1, 2 i - 2}
\end{equation}
for each integer $2 \leq i \leq k$.
Let $V_{k}$ be the subspace of $L^{2}(S^{3})$ spanned by $(v_{k, i})_{i = 1}^{k}$.
Note that the family $(v_{k, 1}, \dots, v_{k, k})$ is an orthonormal basis of $V_{k}$.
The CR Paneitz operator $P(t)$ maps $V_{k}$ to itself.
Let $\calP_{k}(t)$ be the matrix representation of $(1 - t^{2})^{2} P(t)$
with respect to the orthonormal basis $(v_{k, 1}, \dots, v_{k, k})$,
which is a $k \times k$ Hermitian matrix.
By using an explicit formula of $P(t)$,
we can show that $\det \calP_{k}(t) < 0$,
and so $\calP_{k}(t)$ has a negative eigenvalue.
In fact,
$\calP_{k}(t)$ has exactly one negative eigenvalue~\cite{Takeuchi2024-preprint}*{Proposition 5.7}.

\begin{theorem}[\cite{Takeuchi2024-preprint}*{Theorem 1.3}]
\label{thm:infinitely-many-negative-eigenvalue}
	The CR Paneitz operator $P(t)$ on the Rossi sphere $(S^{3}_{t}, T^{1, 0} S^{3}_{t}, \theta)$
	has infinitely many negative eigenvalues counted without multiplicity.
\end{theorem}

\begin{proof}
	As we noted above,
	there exists an eigenfunction $0 \neq f_{k} \in V_{k}$ of $P(t)$ with negative eigenvalue
	for each positive integer $k$.
	Since
	\begin{equation}
		V_{k}
		\subset \bigoplus_{p + q = 2 k - 1} \scrH_{p, q}(S^{3}),
		\qquad
		L^{2}(S^{3}) = \bigoplus_{p, q} \scrH_{p, q}(S^{3}),
	\end{equation}
	the family $(f_{k})_{k = 1}^{\infty}$ is linearly independent.
	This and \cref{thm:spectrum-of-CR-Paneitz} imply that
	$P(t)$ has infinitely many negative eigenvalues without multiplicity.
\end{proof}

\section*{Acknowledgements}
The author is grateful to the referee for a careful reading and some valuable suggestions,
which improves this manuscript.

\bibliography{my-reference,my-reference-preprint}

\end{document}